\documentclass[12pt]{article}

\usepackage{graphicx}
\usepackage{nameref}
\usepackage[shortlabels,inline]{enumitem}
\usepackage{mathtools}
\usepackage{empheq}
\usepackage[shortlabels]{enumitem}
\setlist[enumerate]{nosep}
\usepackage[doc]{optional}
\usepackage{xcolor}
\usepackage[colorlinks=true,
            linkcolor=refkey,
            urlcolor=lblue,
            citecolor=red]{hyperref}
\usepackage{float}
\usepackage{soul}
\usepackage{graphicx}

\definecolor{labelkey}{rgb}{0,0.08,0.45}
\definecolor{refkey}{rgb}{0,0.6,0.0}
\definecolor{Brown}{rgb}{0.45,0.0,0.05}
\definecolor{lime}{rgb}{0.00,0.8,0.0}
\definecolor{lblue}{rgb}{0.5,0.5,0.99}

 \usepackage{mathpazo}

\colorlet{hlcyan}{cyan!30}

\colorlet{hlred}{red!40}

\usepackage{stmaryrd}
\usepackage{amssymb}
\makeatletter
\def\namedlabel#1#2{\begingroup
   \def\@currentlabel{#2}%
   \label{#1}\endgroup
}
\newcommand{\vast}{\bBigg@{4}}
\newcommand{\Vast}{\bBigg@{5}}
\makeatother

\newcommand{\Hc}{\ensuremath{{\mathcal{H}}}}

\newcommand{\bx}{\ensuremath{\mathbf{x}}}
\newcommand{\bz}{\ensuremath{\mathbf{z}}}

\newcommand{\weakly}{\ensuremath{\:{\rightharpoonup}\:}}

\newcommand{\kkk}{\ensuremath{{k\in{\mathbb N}}}}

\newcommand{\menge}[2]{\big\{{#1}~\big |~{#2}\big\}}

\newcommand{\fenv}[1]%
{\ensuremath{\,\overrightarrow{\operatorname{env}}_{#1}}}
\newcommand{\benv}[1]%
{\ensuremath{\,\overleftarrow{\operatorname{env}}_{#1}}}

\newcommand{\scal}[2]{\left\langle{#1},{#2}  \right\rangle}

\newcommand{\RR}{\ensuremath{\mathbb R}}
\newcommand{\ran}{\ensuremath{{\operatorname{ran}}\,}}
\newcommand{\zer}{\ensuremath{\operatorname{zer}}}

\newcommand{\Fix}{\ensuremath{\operatorname{Fix}}}
\newcommand{\Id}{\ensuremath{\operatorname{Id}}}

\newcommand{\tryu}{\ensuremath{T_{\text{\scriptsize Ryu}}}}
\newcommand{\tmt}{\ensuremath{T_{\text{\scriptsize MT}}}}

{\begin{list}{}{%
\settowidth{\labelwidth}{\textrm{#1~}}%
\setlength{\leftmargin}{\labelwidth+\labelsep}}}
{\end{list}}
\usepackage{amsthm}
\usepackage[capitalize,nameinlink]{cleveref}
\crefname{figure}{Figure}{Figures}
\crefname{equation}{}{equations}
\crefname{chapter}{Appendix}{chapters}
\crefname{item}{}{items}
\crefname{enumi}{}{}

\theoremstyle{definition}
\newtheorem{theorem}{Theorem}[section]
\newtheorem{lemma}[theorem]{Lemma}

\newtheorem{corollary}[theorem]{Corollary}

\newtheorem{example}[theorem]{Example}

\newtheorem{fact}[theorem]{Fact}
\newtheorem{remark}[theorem]{Remark}

\DeclarePairedDelimiter{\lVerts}{\lVert}{\rVert}

\newcommand{\norm}[1]{\ensuremath{\lVerts*{#1}}}

\providecommand{\RR}{\mathbb{R}}

\providecommand{\ran}{\operatorname{ran}}

\newcommand{\fix}{\ensuremath{\operatorname{Fix}}}

\providecommand{\Id}{\operatorname{{ Id}}}

\providecommand{\fix}{\operatorname{Fix}}
\providecommand{\ran}{\operatorname{ran}}

\providecommand{\Id}{\operatorname{Id}}

\providecommand{\zer}{\operatorname{zer}}

\providecommand{\RR}{\mathbb{R}}

\definecolor{myblue}{rgb}{.8, .8, 1}

\allowdisplaybreaks 

\begin{document}

\title{\textsf{
The splitting algorithms by~Ryu and by~Malitsky-Tam 
applied to normal cones of linear subspaces
converge strongly to the projection onto the intersection 
}}

\author{
Heinz H.\ Bauschke\thanks{
Mathematics, University
of British Columbia,
Kelowna, B.C.\ V1V~1V7, Canada. E-mail:
\texttt{heinz.bauschke@ubc.ca}.},~
Shambhavi Singh\thanks{
Mathematics, University
of British Columbia,
Kelowna, B.C.\ V1V~1V7, Canada. E-mail:
\texttt{sambha@student.ubc.ca}.},~ 
and
Xianfu Wang\thanks{
Mathematics, University
of British Columbia,
Kelowna, B.C.\ V1V~1V7, Canada. E-mail:
\texttt{shawn.wang@ubc.ca}.}
}

\date{September 22, 2021} 
\maketitle

\vskip 8mm

\begin{abstract} 
Finding a zero of a sum of maximally monotone operators is a fundamental
problem in modern optimization and nonsmooth analysis. 
Assuming that resolvents of the operators are available, this problem
can be tackled with the Douglas-Rachford algorithm.
However, when dealing with three or more operators, one must 
work in a product space with as many factors as there are operators.
In groundbreaking recent work by Ryu and by Malitsky and Tam, it was 
shown that the number of factors can be reduced by one. 
These splitting methods guarantee
\emph{weak} convergence to \emph{some} solution
of the underlying sum problem; strong convergence holds in
the presence of uniform monotonicity.

In this paper, we provide a case study when the operators involved
are normal cone operators of subspaces and the solution set is thus
the intersection of the subspaces. Even though these operators 
lack strict convexity, we show that striking  
conclusions are available in this case: \emph{strong} (instead of weak) 
convergence and
the solution obtained is (not arbitrary but) the \emph{projection onto the intersection}.
Numerical experiments to illustrate our results are also provided.
\end{abstract}

{\small
\noindent
{\bfseries 2020 Mathematics Subject Classification:}
{Primary 
41A50,
49M27,
65K05, 
47H05; 
Secondary 
15A10,
47H09,
49M37,
90C25. 
}

{\small
\noindent {\bfseries Keywords:}
best approximation,
Hilbert space, 
intersection of subspaces, 
linear convergence, 
Malitsky-Tam splitting,
maximally monotone operator, 
nonexpansive mapping, 
resolvent,
Ryu splitting.
}
}

\section{Introduction}

Throughout the paper, we assume that 
\begin{equation}
\text{$X$ is a real Hilbert space}
\end{equation}
with inner product $\scal{\cdot}{\cdot}$ and induced norm $\|\cdot\|$. 
Let $A_1,\ldots,A_n$ be maximally monotone operators on $X$.
(See, e.g., \cite{BC2017} for background on maximally monotone operators.)
One central problem in modern optimization and nonsmooth analysis asks to
\begin{equation}
\label{e:genprob}
\text{find $x\in X$ such that $0\in(A_1+\cdots+A_n)x$.}
\end{equation}
In general, solving \cref{e:genprob} may be quite hard.
Luckily, in many interesting cases, we have access to the 
\emph{firmly nonexpansive resolvents} $J_{A_i} := (\Id+A_i)^{-1}$
which opens the door to employ splitting algorithms to solve 
\cref{e:genprob}. 
The most famous instance is the \emph{Douglas-Rachford algorithm} \cite{DougRach} 
whose importance for this problem was brought to light in the seminal paper 
by Lions and Mercier \cite{LionsMercier}. However, 
the Douglas-Rachford algorithm requires that $n=2$; if $n\geq 3$,
one may employ the Douglas-Rachford algorithm to a reformulation
in the product space $X^n$ \cite[Section~2.2]{Comb09}. 
In recent breakthrough work by Ryu \cite{Ryu}, it was shown 
that for $n=3$ one may formulate an algorithm that works 
in $X^2$ rather than $X^3$. 
We will refer to this method as \emph{Ryu's algorithm}. 
Very recently, Malitsky and Tam proposed in \cite{MT} 
an algorithm for a general $n\geq 3$
that is different from Ryu's and that operators in $X^{n-1}$.
(No algorithms exist in product spaces
featuring fewer factors than $n-1$ factors in a certain technical sense.)
We will review these algorithms in \cref{sec:known} below.
Both Ryu's and the Malitsky-Tam algorithm are known to produce \emph{some} solution 
to \cref{e:genprob} via a sequence that converges \emph{weakly}. 
Strong convergence holds in the presence of uniform monotonicity. 

\emph{The aim of this paper is provide a case study for the situation
when the maximally monotone operators $A_i$ are normal cone operators 
of closed linear subspaces $U_i$ of $X$.}
These operators are not even strictly monotone. 
Our main results show that the splitting algorithms by Ryu and
by Malitsky-Tam actually produce a sequence 
that converges \emph{strongly} and we are able to \emph{identify the limit} to
be the \emph{projection} onto the intersection
$U_1\cap\cdots\cap U_n$!
The proofs of these results rely on the explicit identification of the fixed point
set of the underlying Ryu and Malitsky-Tam operators.
Moreover, a standard translation technique gives the same result for
\emph{affine} subspaces of $X$ provided their intersection is nonempty.

The paper is organized as follows.
In \cref{sec:aux}, we collect various auxiliary results for later use.
The known convergence results on Ryu splitting and on 
Malitsky-Tam splitting are reviewed in \cref{sec:known}.
Our main results are presented in \cref{sec:main}.
Matrix representations of the various operators involved
are provided in \cref{sec:matrix}. 
These are useful for our numerical experiments in \cref{sec:numexp}.
Finally, we offer some concluding remarks in \cref{sec:end}.

The notation employed in this paper is 
standard and follows largely \cite{BC2017}. 
When $z=x+y$ and $x\perp y$, then we also write
$z=x\oplus y$ to stress this fact.
Analogously for the Minkowski sum $Z=X+Y$, we
write $Z= X\oplus Y$ as well as $P_Z = P_X\oplus P_Y$ if $X\perp Y$.

\section{Auxiliary results}
\label{sec:aux}

In this section, we collect useful properties of projection operators
and results on iterating linear/affine nonexpansive operators. 
We start with projection operators.

\subsection{Projections}

\begin{fact}
\label{f:orthoP}
Suppose $U$ and $V$ are nonempty closed convex subsets of $X$ such that
$U\perp V$. Then
$U\oplus V$ is a nonempty closed subset of $X$ and 
\begin{equation}
P_{U\oplus V} = P_U\oplus P_V
\end{equation}
\end{fact}
\begin{proof}
See \cite[Proposition~29.6]{BC2017}. 
\end{proof}

Here is a well known illustration of \cref{f:orthoP} which
we will use repeatedly in the paper (sometimes without explicit mentioning). 

\begin{example}
\label{ex:perp}
Suppose $U$ is a closed linear subspace of $X$.
Then
\begin{equation}
P_{U^\perp} = \Id-P_U.
\end{equation}
\end{example}
\begin{proof}
The orthogonal complement $V := U^\perp$ satisfies $U\perp V$
and also $U+V=X$; thus $P_{U+V}=\Id$ and the result follows.
\end{proof}

\begin{fact}[\bf Anderson-Duffin]
\label{f:AD}
Suppose that $X$ is finite-dimensional and that 
$U,V$ are two linear subspaces of $X$.
Then
\begin{equation}
P_{U\cap V} = 2P_U(P_U+P_V)^\dagger P_V,
\end{equation}
where ``$^\dagger$'' denotes the Moore-Penrose inverse of a matrix.
\end{fact}
\begin{proof}
See, e.g., \cite[Corollary~25.38]{BC2017} or
the original \cite{AD}. 
\end{proof}

\begin{corollary}
\label{c:AD3}
Suppose that $X$ is finite-dimensional and that 
$U,V,W$ are three linear subspaces of $X$.
Then
\begin{equation}
P_{U\cap V\cap W} = 4P_U(P_U+P_V)^\dagger P_V\big(2P_U(P_U+P_V)^\dagger P_V+P_W\big)^\dagger P_W. 
\end{equation}
\end{corollary}
\begin{proof}
Use \cref{f:AD} to find $P_{U\cap V}$,
and then use \cref{f:AD} again on $(U\cap V,W)$. 
\end{proof}

\begin{corollary}
\label{c:ADsum}
Suppose that $X$ is finite-dimensional and that $U,V$ are two linear
subspaces of $X$. 
Then
\begin{subequations}
\begin{align}
P_{U+V} &= \Id-2P_{U^\perp}(P_{U^\perp}+P_{V^\perp})^\dagger P_{V^\perp}\\
&=\Id-2(\Id-P_U)\big(2\Id-P_U-P_V \big)^\dagger(\Id-P_V).
\end{align}
\end{subequations}
\end{corollary}
\begin{proof}
Indeed, $U+V = (U^\perp\cap V^\perp)^\perp$ and so 
$P_{U+V} = \Id - P_{U^\perp\cap V^\perp}$.
Now apply \cref{f:AD} to $(U^\perp,V^\perp)$
followed by \cref{ex:perp}. 
\end{proof}

\begin{fact}
\label{f:Pran}
Let $Y$ be a real Hilbert space, and 
let $A\colon X\to Y$ be a continuous linear operator with closed range.
Then 
\begin{equation}
P_{\ran A} = AA^\dagger. 
\end{equation}
\end{fact}
\begin{proof}
See, e.g., \cite[Proposition~3.30(ii)]{BC2017}. 
\end{proof}

\subsection{Linear (and affine) nonexpansive iterations}

We now turn results on iterating linear or affine nonexpansive operators.

\begin{fact}
\label{f:convasymp}
Let $L\colon X\to X$ be linear and nonexpansive, 
and let $x\in X$. 
Then 
\begin{equation}
L^kx \to P_{\Fix L}(x)
\quad\Leftrightarrow\quad
L^kx - L^{k+1}x \to 0. 
\end{equation}
\end{fact}
\begin{proof}
See \cite[Proposition~4]{Baillon},
\cite[Theorem~1.1]{BBR}, 
\cite[Theorem~2.2]{BDHP}, or 
\cite[Proposition~5.28]{BC2017}. 
(The versions in \cite{Baillon} and \cite{BBR} are much more general.)
\end{proof}

\begin{fact}
\label{f:averasymp}
Let $T\colon X\to X$ be averaged nonexpansive with 
$\Fix T\neq\varnothing$. 
Then $(\forall x\in X)$ 
$T^kx-T^{k+1}x\to 0$. 
\end{fact}
\begin{proof}
See Bruck and Reich's paper \cite{BruRei} or \cite[Corollary~5.16(ii)]{BC2017}.
\end{proof}

\begin{corollary}
\label{c:key}
Let $L\colon \Hc\to\Hc$ be linear and averaged nonexpansive.
Then 
\begin{equation}
(\forall x\in\Hc)\quad L^kx \to P_{\Fix L}(x).
\end{equation}
\end{corollary}
\begin{proof}
Because $0\in\Fix L$, we have $\Fix L\neq\varnothing$. 
Now combine
\cref{f:convasymp} with \cref{f:averasymp}.
\end{proof}

\begin{fact}
\label{f:BLM}
Let $L$ be a linear nonexpansive operator and let $b\in X$.
Set $T\colon X\to X\colon x\to Lx+b$ and suppose that 
$\Fix T\neq\varnothing$.
Then $b\in\ran(\Id-L)$, and for every $x\in X$ and
$a\in(\Id-L)^{-1}b$, the following hold:
\begin{enumerate}
\item $b=a-La\in\ran(\Id-L)$.
\item $\Fix T = a + \Fix L$.
\item 
\label{f:BLMiii}
$P_{\Fix T}(x) = P_{\Fix L}(x) + P_{(\Fix L)^\perp}(a)$.
\item $T^kx = L^k(x-a)+a$.
\item $L^kx \to P_{\Fix L}x$ $\Leftrightarrow$ $T^kx \to P_{\Fix T}x$. 
\end{enumerate}
\end{fact}
\begin{proof}
See \cite[Lemma~3.2 and Theorem~3.3]{BLM}. 
\end{proof}

\begin{remark}
\label{r:BLM}
Consider \cref{f:BLM} and its notation.
If $a\in(\Id-L)^{-1}b$ then 
$P_{(\Fix L)^\perp}$ is likewise because 
$b = (\Id-L)a = (\Id-L)(P_{\Fix L}(a)+P_{(\Fix L)^\perp}(a))
=P_{(\Fix L)^\perp}(a)$; moreover,
using \cite[Lemma~3.2.1]{Groetsch}, we see that 
\begin{equation}
(\Id-L)^\dagger b 
=(\Id-L)^\dagger(\Id-L)a
=P_{(\ker(\Id-L))^\perp}(a)
=P_{(\Fix L)^\perp}(a),
\end{equation}
where again ``$^\dagger$'' denotes the Moore-Penrose inverse of a continuous
linear operator (with possibly nonclosed range).
So given $b\in X$, we may concretely set 
\begin{equation}
a = (\Id-L)^\dagger b \in (\Id-L)^{-1}b;
\end{equation}
with this choice, \cref{f:BLMiii} turns into the even more pleasing identity
\begin{equation}
P_{\Fix T}(x) = P_{\Fix L}(x) + a.
\end{equation}
\end{remark}

\section{Known results on Ryu and on Malitsky-Tam splitting}

In this section, we present the precise form of
Ryu's and the Malitsky-Tam algorithms and review 
known convergence results. 

\label{sec:known}

\subsection{Ryu splitting}

We start with Ryu's algorithm. In this subsection, 
\begin{equation}
\text{$A,B,C$ are maximally monotone operators on 
$X$, 
}
\end{equation}
with resolvents $J_A,J_B,J_C$, respectively.

The problem of interest is to 
\begin{equation}
\label{e:Ryuprob}
\text{find $x\in X$ such that $0\in (A+B+C)x$,}
\end{equation}
and we assume that \cref{e:Ryuprob} has a solution. 
The algorithm pioneered by Ryu  \cite{Ryu} provides a method
for finding a solution to \cref{e:Ryuprob}.
It proceeds as follows. Set\footnote{We will express vectors in
product spaces both as column and as row vectors depending 
on which version is more readable.}
\begin{equation}
\label{e:genM}
M\colon X\times X\to X\times X\times X\colon
\begin{pmatrix}
x\\
y
\end{pmatrix}
\mapsto
\begin{pmatrix}
J_A(x)\\[+1mm]
J_B(J_A(x)+y)\\[+1mm]
J_C\big(J_A(x)-x+J_B(J_A(x)+y)-y\big) 
\end{pmatrix}. 
\end{equation}
Next, denote by $Q_1\colon X\times X\times X\to X\colon 
(x_1,x_2,x_3)\mapsto x_1$
and similarly for $Q_2$ and $Q_3$. 
We also set $\Delta := \menge{(x,x,x)\in X^3}{x\in X}$. 
We are now ready to introduce the \emph{Ryu operator}
\begin{equation}
\label{e:TRyu}
T := \tryu \colon X^2\to X^2\colon 
z\mapsto 
z + \big((Q_3-Q_1)Mz,(Q_3-Q_2)Mz\big). 
\end{equation}
Given a starting point $(x_0,y_0)\in X\times X$, 
the basic form of 
Ryu splitting generates a governing sequence via
\begin{equation}
\label{e:basicRyu}
(\forall \kkk)\quad (x_{k+1},y_{k+1}) := 
(1-\lambda)(x_k,y_k) + \lambda T(x_k,y_k).
\end{equation}

The following result records the basic convergence properties 
by Ryu \cite{Ryu}, 
and recently improved by Arag\'on-Artacho, Campoy, and Tam \cite{AACT20}. 

\begin{fact}[\bf Ryu and also Aragon-Artacho-Campoy-Tam]
\label{f:Ryu}
The operator
$\tryu$ is nonexpansive with 
\begin{equation}
\label{e:fixtryu}
\Fix \tryu = 
\menge{(x,y)\in X\times X}{J_A(x)=J_B(J_A(x)+y) = J_C(R_A(x)-y)}
\end{equation}
and 
\begin{equation}
\zer(A+B+C) = J_A\big(Q_1\Fix\tryu\big). 
\end{equation}
Suppose that $0<\lambda<1$ and 
consider the sequence generated by 
\cref{e:basicRyu}. 
Then there exists $(\bar{x},\bar{y})\in X\times X$ such that 
\begin{equation}
(x_k,y_k)\weakly (\bar{x},\bar{y}) \in \Fix\tryu, 
\end{equation}
\begin{equation}
M(x_k,y_k) \weakly M(\bar{x},\bar{y})\in\Delta, 
\end{equation}
and 
\begin{equation}
\label{e:210815d}
\big((Q_3-Q_1)M(x_k,y_k),(Q_3-Q_2)M(x_k,y_k)\big)\to (0,0). 
\end{equation}
In particular,
\begin{equation}
J_A(x_k) \weakly J_A\bar{x}\in \zer(A+B+C).
\end{equation}
\end{fact}
\begin{proof}
See \cite{Ryu} and \cite{AACT20}. 
\end{proof}

\subsection{Malitsky-Tam splitting}

We now turn to the Malitsky-Tam algorithm. 
In this subsection, let $n\in\{3,4,\ldots\}$ and 
let $A_1,A_2,\ldots,A_n$ 
be maximally monotone operators on  $X$. 
The problem of interest is to 
\begin{equation}
\label{e:MTprob}
\text{find $x\in X$ such that $0\in (A_1+A_2+\cdots + A_n)x$,}
\end{equation}
and we assume that \cref{e:MTprob} has a solution. 
The algorithm proposed by Malitsky and Tam \cite{MT} provides a method
for finding a solution to \cref{e:MTprob}.
Now set\footnote{Again, we will express vectors in
product spaces both as column and as row vectors depending 
on which version is more readable.}
\begin{subequations}
\label{e:MTM}
\begin{align}
M\colon X^{n-1}
&\to 
X^n\colon
\begin{pmatrix}
z_1\\
\vdots\\
z_{n-1}
\end{pmatrix}
\mapsto
\begin{pmatrix}
x_1\\
\vdots\\
x_{n-1}\\
x_n
\end{pmatrix},
\quad\text{where}\;\;
\\[+2mm]
&(\forall i\in\{1,\ldots,n\})
\;\;
x_i = \begin{cases}
J_{A_1}(z_1), &\text{if $i=1$;}\\
J_{A_i}(x_{i-1}+z_i-z_{i-1}), &\text{if $2\leq i\leq n-1$;}\\
J_{A_n}(x_1+x_{n-1}-z_{n -1}), &\text{if $i=n$.}
\end{cases}
\end{align}
\end{subequations}
As before, we denote by $Q_1\colon X^n \to X\colon 
(x_1,\ldots,x_{n-1},x_n)\mapsto x_1$
and similarly for $Q_2,\ldots Q_n$. 
We also set $\Delta := \menge{(x,\ldots,x)\in X^n}{x\in X}$, 
the diagonal in $X^n$. 
We are now ready to introduce the \emph{Malitsky-Tam (MT) operator}
\begin{equation}
\label{e:TMT}
T := \tmt \colon X^{n-1}\to X^{n-1}\colon 
\bz\mapsto 
\bz+ 
\begin{pmatrix}
(Q_2-Q_1)M\bz\\
(Q_3-Q_2)M\bz\\
\vdots\\
(Q_n-Q_{n-1})M\bz
\end{pmatrix}. 
\end{equation}
Given a starting point $\bz_0 \in X^{n-1}$, 
the basic form of 
MT splitting generates a governing sequence via
\begin{equation}
\label{e:basicMT}
(\forall \kkk)\quad \bz_{k+1} := 
(1-\lambda)\bz_{k} + \lambda T\bz_k. 
\end{equation}

The following result records the basic convergence. 

\begin{fact}[\bf Malitsky-Tam]
\label{f:MT}
The operator
$\tmt$ is nonexpansive with 
\begin{equation}
\label{e:fixtmt}
\Fix \tmt = 
\menge{z\in X^{n-1}}{Mz \in \Delta},
\end{equation}
\begin{equation}
\zer(A_1+\cdots+A_n) = J_{A_1}\big(Q_1\Fix\tmt\big).
\end{equation}
Suppose that $0<\lambda<1$ and 
consider the sequence generated by 
\cref{e:basicMT}. 
Then there exists $\bar{\bz}\in X^{n-1}$ such that 
\begin{equation}
\bz_k\weakly \bar{\bz} \in \Fix\tmt, 
\end{equation}
\begin{equation}
M\bz_{k} \weakly M\bar{\bz}\in\Delta,
\end{equation}
and 
\begin{equation}
\label{e:210817a}
(\forall (i,j)\in\{1,\cdots,n\}^2)\quad 
(Q_i-Q_j)M\bz_k\to 0. 
\end{equation}
In particular,
\begin{equation}
J_{A_1}Q_1 M\bz_k \weakly J_AQ_1 M\bar{\bz}\in \zer(A_1+\ldots+A_n). 
\end{equation}
\end{fact}
\begin{proof}
See \cite{MT}. 
\end{proof}

\section{Main Results}

We are now ready to tackle our main results.
We shall find useful descriptions of the fixed point sets
of the Ryu and the Malitsky-Tam operators.
These description will allow us to deduce strong convergence of the 
iterates to the projection
onto the intersection. 

\label{sec:main}

\subsection{Ryu splitting}

In this subsection, we assume that 
\begin{equation}
\text{
$U,V,W$ are closed linear subspaces of $X$.
}
\end{equation}
We set
\begin{equation}
A := N_U,
\;\;
B := N_V,
\;\;
C := N_{W}. 
\end{equation}
Then
\begin{equation}
Z := \zer(A+B+C) = U\cap V\cap W.
\end{equation}
Using linearity of the projection operators, the operator $M$ defined in \cref{e:genM} turns into
\begin{equation}
\label{e:linM}
M\colon X\times X\to X\times X\times X\colon
\begin{pmatrix}
x\\
y
\end{pmatrix}
\mapsto
\begin{pmatrix}
P_Ux\\[+1mm]
P_VP_Ux+P_Vy\\[+1mm]
P_WP_Ux + \textcolor{black}{P_W}P_VP_Ux-P_Wx \textcolor{black}{+}P_WP_Vy-P_Wy
\end{pmatrix},
\end{equation}
while the Ryu operator is still (see \cref{e:TRyu})
\begin{equation}
\label{e:linT}
T := \tryu \colon X^2\to X^2\colon 
z\mapsto 
z + \big((Q_3-Q_1)Mz,(Q_3-Q_2)Mz\big). 
\end{equation}

We now determine the fixed point set of the Ryu operator.

\begin{lemma}
\label{firelemma}
Let $(x,y)\in X\times X$. 
Then 
\begin{equation}
\label{e:210815b}
\Fix T = \big(Z\times\{0\}\big) \oplus 
\Big(\big(U^\perp\times V^\perp) \cap \big(\Delta^\perp+(\{0\}\times W^\perp) \big)\Big),
\end{equation}
where $\Delta = \menge{(x,x)\in X\times X}{x\in X}$. 
Consequently, setting 
\begin{equation}
E=\big(U^\perp\times V^\perp) \cap \big(\Delta^\perp+(\{0\}\times W^\perp) \big),
\end{equation}
we have 
\begin{equation}
\label{e:210815c}
P_{\Fix T}(x,y) = (P_{Z}x,0)\oplus P_E(x,y) \in (P_Zx\oplus U^\perp)\times V^\perp. 
\end{equation}
\end{lemma}
\begin{proof}
Note that $(x,y)=(P_{W^\perp}y+(x-P_{W^\perp}y),P_{W^\perp}y+P_Wy)
=(P_{W^\perp}y,P_{W^\perp}y)+(x-P_{W^\perp}y,P_Wy)
\in \Delta + (X\times W)$.
Hence 
\begin{equation}
X\times X = \Delta+(X\times W)\;\;\text{is closed;}
\end{equation}
consequently, by, e.g., \cite[Corollary~15.35]{BC2017}, 
\begin{equation}
\label{e:210815a}
\Delta^\perp + (\{0\}\times W^\perp)\;\;\text{is closed.}
\end{equation}
Next, using \cref{e:fixtryu}, we have the equivalences 
\begin{subequations}
\begin{align}
&\hspace{-1cm}(x,y)\in\Fix\tryu\\
&\Leftrightarrow
P_Ux = P_V\big(P_Ux+y\big) = P_W\big(R_Ux-y\big)\\
&\Leftrightarrow
P_Ux\in Z 
\;\land\; 
y\in V^\perp
\;\land\; 
P_Ux = P_W\big(P_Ux-P_{U^\perp}x-y\big)\\
&\Leftrightarrow
x\in Z + U^\perp 
\;\land\; 
y\in V^\perp
\;\land\; 
P_{U^\perp}x + y \in W^\perp. 
\end{align}
\end{subequations}
Now define the linear operator 
\begin{equation}
S\colon X\times X\to X\colon (x,y)\mapsto x+y.
\end{equation}
Hence 
\begin{subequations}
\begin{align}
\Fix \tryu
&= 
\menge{(x,y)\in (Z+U^\perp)\times V^\perp}{P_{U^\perp}x+y\in W^\perp}\\
&=
\menge{(z+u^\perp,v^\perp)}{z\in Z,\,u^\perp\in U^\perp,\,v^\perp\in V^\perp, 
\,u^\perp + v^\perp\in W^\perp}\\
&=(Z\times\{0\})
\oplus \big((U^\perp\times V^\perp) \cap S^{-1}(W^\perp)\big). 
\end{align}
\end{subequations}
On the other hand,
$S^{-1}(W^\perp) = (\{0\}\times W^\perp)+\ker S
= (\{0\}\times W^\perp)+\Delta^\perp$ is closed by \cref{e:210815a}.
Altogether,
\begin{equation}
\Fix \tryu
= (Z\times\{0\})
\oplus \big((U^\perp\times V^\perp) \cap 
((\{0\}\times W^\perp)+\Delta^\perp)\big), 
\end{equation}
i.e., 
\cref{e:210815b} holds. 
Finally, \cref{e:210815c} follows from \cref{f:orthoP}.
\end{proof}

We are now ready for the main convergence result on Ryu's algorithm.

\begin{theorem}[\bf main result on Ryu splitting]
Given $0<\lambda<1$ and $(x_0,y_0)\in X\times X$, generated
the sequence $(x_k,y_k)_\kkk$ via\footnote{Recall \cref{e:linM} and \cref{e:linT}
for the definitions of $M$ and $T$.}
\begin{equation}
(\forall\kkk)\quad (x_{k+1},y_{k+1}) = (1-\lambda)(x_k,y_k)+\lambda T(x_k,y_k).
\end{equation}
Then 
\begin{equation}
\label{e:210815e}
M(x_k,y_k)\to \big(P_Z(x_0),P_Z(x_0),P_Z(x_0)\big);
\end{equation}
in particular,
\begin{equation}
P_U(x_k)\to P_Z(x_0).
\end{equation}
\end{theorem}
\begin{proof}
Set $T_\lambda := (1-\lambda)\Id+\lambda T$ and observe that 
$(x_k,y_k)_\kkk = (T^k_\lambda(x_0,y_0))_\kkk$. 
Hence, by \cref{c:key} and \cref{e:210815c}
\begin{subequations}
\begin{align}
(x_k,y_k)&\to P_{\Fix T_\lambda}(x_0,y_0)=P_{\Fix T}(x_0,y_0)\\
&=(P_Zx_0,0)+P_E(x_0,y_0)\in (P_Zx_0\oplus U^\perp)\times V^\perp,
\end{align}
\end{subequations}
where $E$ is as in \cref{firelemma}. 
Hence 
\begin{equation}
Q_1M(x_k,y_k) = P_Ux_k \to P_U(P_Zx_0) = P_Zx_0. 
\end{equation}
Now \cref{e:210815d} yields
\begin{equation}
\lim_{k\to\infty}Q_1M(x_k,y_k)=\lim_{k\to\infty}Q_2M(x_k,y_k)
= \lim_{k\to\infty}Q_3M(x_k,y_k)= P_Zx_0, 
\end{equation}
i.e., \cref{e:210815e} and we're done.
\end{proof}

\subsection{Malitsky-Tam splitting}

\label{ss:MTlin}

Let $n\in\{3,4,\ldots\}$. 
In this subsection, we assume that 
$U_1,\ldots,U_n$ are closed linear subspaces of $X$.
We set
\begin{equation}
(\forall i\in\{1,2,\ldots,n\})
\quad 
A_i := N_{U_i} \;\;\text{and}\;\;
P_i := P_{U_i}. 
\end{equation}
Then
\begin{equation}
Z := \zer(A_1+\cdots+A_n) = U_1\cap \cdots \cap U_n.
\end{equation}
The operator $M$ defined in \cref{e:MTM} turns into
\begin{subequations}
\label{e:linMTM}
\begin{align}
M\colon X^{n-1}
&\to 
X^n\colon
\begin{pmatrix}
z_1\\
\vdots\\
z_{n-1}
\end{pmatrix}
\mapsto
\begin{pmatrix}
x_1\\
\vdots\\
x_{n-1}\\
x_n
\end{pmatrix},
\quad\text{where}\;\;
\\[+2mm]
&(\forall i\in\{1,\ldots,n\})
\;\;
x_i = \begin{cases}
P_{1}(z_1), &\text{if $i=1$;}\\
P_{i}(x_{i-1}+z_i-z_{i-1}), &\text{if $2\leq i\leq n-1$;}\\
P_{n}(x_1+x_{n-1}-z_{n -1}), &\text{if $i=n$.}
\end{cases}
\end{align}
\end{subequations}
and the MT operator remains (see \cref{e:TMT}) 
\begin{equation}
\label{e:linMT}
T := \tmt \colon X^{n-1}\to X^{n-1}\colon 
\bz\mapsto 
\bz+ 
\begin{pmatrix}
(Q_2-Q_1)M\bz\\
(Q_3-Q_2)M\bz\\
\vdots\\
(Q_n-Q_{n-1})M\bz
\end{pmatrix}. 
\end{equation}

We now determine the fixed point set of the Malitsky-Tam operator. 

\begin{lemma}
\label{mtfixlemma}
The fixed point set of the MT operator $T=\tmt$ is 
\begin{align}
\label{e:210817e}
\Fix T &= \menge{(z,\ldots,z)\in X^{n-1}}{z\in Z} \oplus E,
\end{align}
where 
\begin{subequations}
\label{e:bloodyE}
\begin{align}
E &:= 
\ran\Psi \cap \big(X^{n-2}\times U_n^\perp)\\
&\subseteq 
U_1^\perp \times 
\cdots
\times (U_1^\perp+\cdots+U_{n-2}^\perp)
\times \big((U_1^\perp+\cdots+U_{n-1}^\perp)\cap U_n^\perp\big)
\end{align}
\end{subequations}
and 
\begin{subequations}
\label{e:bloodyPsi}
\begin{align}
\Psi\colon U_1^\perp \times \cdots \times U_{n-1}^\perp
&\to X^{n-1}\\
(y_1,\ldots,y_{n-1})&\mapsto 
(y_1,y_1+y_2,\ldots,y_1+y_2+\cdots+y_{n-1})
\end{align}
\end{subequations}
 is the continuous linear partial sum operator which has closed range. 

Let $\bz = (z_1,\ldots,z_{n-1})\in X^{n-1}$,
and set $\bar{z} = (z_1+z_2+\cdots+z_{n-1})/(n-1)$.
Then
\begin{equation}
\label{e:210817d}
P_{\Fix T}\bz = (P_Z\bar{z},\ldots,P_Z\bar{z}) \oplus P_E\bz \in X^{n-1}
\end{equation}
and hence 
\begin{equation}
\label{e:210817g}
P_1(Q_1P_{\Fix T})\bz = P_Z\bar{z}.
\end{equation}

\end{lemma}
\begin{proof}
Assume temporarily that $\bz\in\Fix T$ and 
set $\bx = M\bz = (x_1,\ldots,x_n)$. 
Then $\bar{x} := x_1=\cdots= x_n$ and so $\bar{x}\in Z$.
Now $P_1z_1=x_1=\bar{x}\in Z$ and thus 
\begin{equation}
z_1 \in \bar{x}+U_1^\perp \subseteq Z + U_1^\perp.
\end{equation}
Next, 
$\bar{x}=x_2 = P_2(x_1+z_2-z_1)=
P_2x_1 + P_2(z_2 - z_1)=
P_2\bar{x}+P_2(z_2-z_1) = \bar{x}$,
which implies $P_2(z_2-z_1)=0$ and so 
$z_2-z_1\in U_2^\perp$. 
It follows that 
\begin{equation}
z_2 \in z_1+U_2^\perp.
\end{equation}
Similarly, by considering $x_3,\ldots,x_{n-1}$, we obtain 
\begin{equation}
\label{e:210817b}
z_3\in z_2 + U_3^\perp, \ldots, z_{n-1}\in z_{n-2}+U_{n-1}^\perp.
\end{equation}
Finally, 
$\bar{x}=x_n=P_n(x_1+x_{n-1}-z_{n-1})=P_n(\bar{x}+\bar{x}-z_{n-1})
=2\bar{x}-P_nz_{n-1}$,
which implies $P_nz_{n-1}=\bar{x}$, i.e., 
$z_{n-1}\in \bar{x}+U_n^\perp$. 
Combining with \cref{e:210817b}, we see that 
$z_{n-1}$ satisfies
\begin{equation}
z_{n-1}\in (z_{n-2}+U_{n-1}^\perp)\cap(P_1z_1+U_n^\perp).
\end{equation}
To sum up, our $\bz\in \Fix T$ must satisfy
\begin{subequations}
\label{e:210817c}
\begin{align}
z_1 &\in Z + U_1^\perp\\
z_2 &\in z_1+U_2^\perp\\
&\;\;\vdots\\
z_{n-2}&\in z_{n-3}+U_{n-2}^\perp\\
z_{n-1}&\in (z_{n-2}+U_{n-1}^\perp)\cap(P_1z_1+U_n^\perp). 
\end{align}
\end{subequations}

We now show the converse. 
To this end, assume now that our $\bz$ satisfies \cref{e:210817c}. 
Note that $Z^\perp = \overline{U_1^\perp+\cdots + U_n^\perp}$.
Because $z_1 \in Z+U_1^\perp$,
there exists $z\in Z$ and $u_1^\perp\in U_1^\perp$ such that 
$z_1 = z\oplus u_1^\perp$.
Hence $x_1=P_1z_1 = P_1z = z$. 
Next, 
$z_2\in z_1+U_2^\perp$, say 
$z_2=z_1+u_2^\perp = z\oplus(u_1^\perp+u_2^\perp)$,
where $u_2^\perp \in U_2^\perp$. 
Then 
$x_2 = P_2(x_1+z_2-z_1)=P_2(z+u_2^\perp)=P_2z = z$. 
Similarly, there exists also $u_3^\perp\in U_3^\perp,
\ldots,u_{n-1}^\perp\in U_{n-1}^\perp$ such that 
$x_3=\cdots=x_{n-1}=z$ and 
$z_i=z\oplus(u_1^\perp+\cdots +u_i^\perp)$ for 
$2\leq i\leq n-1$. 
Finally, we also have $z_{n-1}=z\oplus u_n^\perp$
for some $u_n^\perp\in U_n^\perp$. 
Thus
$x_n = P_n(x_1+x_{n-1}-z_{n-1})=P_n(2z-(z+u_{n}^\perp))
=P_nz = z$. Altogether, $\bz \in \Fix T$.
We have thus verified the description of $\Fix T$ 
announced in \cref{e:210817e}, using the convenient
notation of the operator $\Psi$ which is easily seen to have closed range.

Next, we observe that 
\begin{equation}
\label{e:210817c+}
D := \menge{(z,\ldots,z)\in X^{n-1}}{z\in Z} = 
Z^{n-1}\cap \Delta,
\end{equation}
where $\Delta$ is the diagonal in $X^{n-1}$ which has projection
$P_\Delta(z_1,\ldots,z_n)=(\bar{z},\ldots,\bar{z})$ 
(see, e.g., \cite[Proposition~26.4]{BC2017}). 
By convexity of $Z$, we clearly have 
$P_\Delta(Z^{n-1})\subseteq Z^{n-1}$.
Because $Z^{n-1}$ is a closed linear subspace of $X^{n-1}$,
\cite[Lemma~9.2]{Deutsch} and \cref{e:210817c+} yield
$P_D = P_{Z^{n-1}}P_\Delta$ and therefore
\begin{equation}
\label{e:210817f}
P_D\bz
= P_{Z^{n-1}}P_\Delta\bz
= \big(P_Z\bar{z},\ldots,P_Z\bar{z}\big). 
\end{equation}

Combining \cref{e:210817e}, \cref{f:orthoP}, \cref{e:210817c+},
and \cref{e:210817f} yields \cref{e:210817d}. 

Finally, observe that $Q_1(P_E\bz)\in U_1^\perp$ by 
\cref{e:bloodyE}.
Thus $Q_1(P_{\Fix T}\bz)\in P_Z\bar{z} + U_1^\perp$ 
and therefore \cref{e:210817g} follows.
\end{proof}

We are now ready for the main convergence result on 
the Malitsky-Tam algorithm.

\begin{theorem}[\bf main result on Malitsky-Tam splitting]
Given $0<\lambda<1$ and $\bz_0=(z_{0,1},\ldots,z_{0,n-1}) \in X^{n-1}$, 
generate the sequence $(\bz_k)_\kkk$ via\footnote{Recall \cref{e:linMTM} 
and \cref{e:linMT} for the definitions of $M$ and $T$.}
\begin{equation}
(\forall\kkk)\quad
\bz_{k+1} = (1-\lambda)\bz_k + \lambda T\bz_k.
\end{equation}
Set 
\begin{equation}
p := \frac{1}{n-1}\big(z_{0,1}+\cdots+z_{0,n-1}\big). 
\end{equation}
Then there exists $\bar{\bz}\in X^{n-1}$ such that 
\begin{equation}
\bar{\bz}_k \to \bar{\bz} \in \Fix T,
\end{equation}
and 
\begin{equation}
\label{e:210817h}
M\bz_k \to M\bar{\bz} = (P_Zp,\ldots,P_Zp) \in X^n. 
\end{equation}
In particular, 
\begin{equation}
\label{e:210817i}
P_1(Q_1\bz_k)= Q_1M\bz_k \to P_Z(p) = \tfrac{1}{n-1}P_Z\big(z_{0,1}+\cdots+z_{0,n-1}\big).
\end{equation}
Consequently, if $x_0\in X$ and 
$\bz_0 = (x_0,\ldots,x_0)\in X^{n-1}$,
then 
\begin{equation}
\label{e:yayyay}
P_1Q_1\bz_k \to P_Zx_0.
\end{equation}
\end{theorem}
\begin{proof}
Set $T_\lambda := (1-\lambda)\Id+\lambda T$ and observe that 
$(\bz_k)_\kkk = (T_\lambda^k\bz)_\kkk$. 
Hence, by \cref{c:key} and \cref{mtfixlemma},
\begin{subequations}
\begin{align}
\bz_k&\to P_{\Fix T_\lambda}\bz_0=P_{\Fix T}\bz_0\\
&=(P_Zp,\ldots,P_Zp)\oplus P_E(\bz_0), 
\end{align}
\end{subequations}
where $E$ is as in \cref{mtfixlemma}. 
Hence, using also \cref{e:210817g}, 
\begin{subequations}
\begin{align}
Q_1M\bz_k &= P_1Q_1\bz_k \\
&\to P_1Q_1\big((P_Zp,\ldots,P_Zp)\oplus P_E(\bz_0)\big)\\
&=P_1\big(P_Zp+Q_1(P_E(\bz_0))\big)\\
&\in P_1\big(P_Zp+U_1^\perp\big)\\
&=\{P_1P_Zp\}\\
&=\{P_Zp\},
\end{align}
\end{subequations}
i.e., $Q_1M\bz_k\to P_Zp$. 
Now \cref{e:210817a} yields 
$Q_iM\bz_k\to P_zp$ for every $i\in\{1,\ldots,n\}$.
This yields \cref{e:210817h} and \cref{e:210817i}.

Finally, the ``Consequently'' part is clear because when 
$\bz_0$ has this special form, then $p=x_0$. 
\end{proof}

\subsection{Extension to the consistent affine case} 

In this subsection, we comment on the behaviour of the 
splitting algorithms by  Ryu and 
by Malitsky-Tam in the consistent affine case.
To this end, we shall assume that 
$V_1,\ldots,V_n$ are closed affine subspaces of $X$ with
nonempty intersection:
\begin{equation}
V := V_1\cap V_2\cap \cdots \cap V_n \neq \varnothing.
\end{equation}
We repose the problem of finding a point in $Z$ as 
\begin{equation}
\text{find $x\in X$ such that $0\in (A_1+A_2+\cdots+A_n)x$,}
\end{equation}
where each $A_i = N_{V_i}$.
When we consider Ryu splitting, we also impose $n=3$.
Set $U_i:=V_i-V_i$, which is the \emph{parallel space} of $V_i$.
Now let $v\in V$.
Then $V_i = v+U_{i}$ and hence 
$J_{N_{V_i}}=P_{V_i} = P_{v+U_i}$ satisfies
$P_{v+U_i} = v+P_{U_i}(x-v)=P_{U_i}x+P_{U_i^\perp}(v)$.
Put differently, the resolvents from the affine problem
are translations of the the resolvents from the corresponding linear problem
which considers $U_i$ instead of $V_i$.

The construction of the operator $T\in \{\tryu,\tmt\}$ now shows 
that it is a translation of the corresponding operator from the linear problem.
And finally $T_\lambda = (1-\lambda)\Id+\lambda T$ is a translation of 
the corresponding operator from the linear problem which we denote by $L_\lambda$:
$L_\lambda = (1-\lambda)\Id+\lambda L$, where $L$ is 
either the Ryu operator of the Malitsky-Tam operator of the parallel linear problem,
and  there exists $b\in X^{n-1}$ such that 
\begin{equation}
T_\lambda(x) = L_\lambda(x)+b. 
\end{equation}
By \cref{f:BLM} (applied in $X^{n-1}$), 
there exists a vector $a\in X^{n-1}$ such that 
\begin{equation}
\label{e:para1}
(\forall\kkk)\quad T_\lambda^kx = a + L_\lambda^k(x-a).
\end{equation}
In other words, the behaviour in the affine case is essentially
the same as in the linear parallel case, appropriately shifted by the vector $a$.
Moreover, because $L_\lambda^k\to P_{\Fix L}$ in the parallel linear setting,
we deduce from \cref{f:BLM} that 
\begin{equation}
T_\lambda^k \to P_{\Fix T}
\end{equation}
By \cref{e:para1}, the rate of convergence in the affine case are identical
to the rate of convergence in the parallel linear case. 
Thus, if $(x_k,y_k)_\kkk$ is the governing sequence generated
by Ryu splitting, then 
\begin{equation}
P_{V_1}x_k\to P_V(x_0).
\end{equation}
And if $\bz_k=(z_{k,1},\ldots,z_{k,n-1})_\kkk$ is the sequence 
generated by Malitsky-Tam splitting, then 
\begin{equation}
P_{V_1}Q_1\bz_k\to \tfrac{1}{n-1}P_V(z_{0,1}+\cdots+z_{0,n-1}).
\end{equation}

To sum up this subsection, we note that 
\emph{in the consistent affine case, Ryu's and the Malitsky-Tam algorithm
exhibit the same pleasant convergence behaviour as their linear parallel counterparts!}

It is, however, presently quite unclear how these two algorithms 
behave  when $V=\varnothing$. 

\section{Matrix representation}

\label{sec:matrix}

In this section, we assume that $X$ is finite-dimensional,
say 
\begin{equation}
X=\RR^d.
\end{equation}
The two splitting algorithms considered in this paper are of the form 
\begin{equation}
\label{e:210818a}
T_\lambda^k\to P_{\Fix T}, 
\quad\text{where $0<\lambda<1$ and $T_\lambda = (1-\lambda)\Id+\lambda T$.}
\end{equation}
Note that $T$ is a linear operator; hence, so is $T_\lambda$ and 
by \cite[Corollary~2.8]{BLM}, the convergence in \cref{e:210818a} 
is \emph{linear} because $X$ is finite-dimensional. 
What can be said about this rate?
By \cite[Theorem~2.12(ii) and Theorem~2.18]{BBCNPW2}, a (sharp) 
\emph{lower bound} for the rate of linear 
convergence is the \emph{spectral radius} of $T_\lambda - P_{\Fix T}$, i.e., 
\begin{equation}
\rho\big(T_\lambda- P_{\Fix T}\big) =
\max \big|\{\text{spectral values of $T_\lambda- P_{\Fix T}$}\}\big|,
\end{equation}
while an \emph{upper bound} is the operator norm
\begin{equation}
\big\|T_\lambda- P_{\Fix T}\big\|. 
\end{equation}
The lower bound is optimal and close to the true rate of convergence, 
see \cite[Theorem~2.12(i)]{BBCNPW2}.
Both spectral radius and operator norms of matrices are available
in programming languages such as \texttt{Julia} \cite{Julia} which features 
strong numerical linear algebra capabilities.
In order to compute these bounds for the linear rates, 
we must provide \emph{matrix representations} 
for $T$ (which immediately gives rise to one for $T_\lambda$) and for $P_{\Fix T}$.
In the previous sections, we casually switched back and forth
being column and row vector representations for readability.
In this section, we need to get the structure of the objects right.
To visually stress this, we will use \emph{square brackets} for vectors and 
matrices.

For the remainder of this section, we fix 
three linear subspaces $U,V,W$ of $\RR^d$, with intersection 
\begin{equation}
Z = U\cap V\cap W. 
\end{equation}

We assume that the matrices $P_U,P_V,P_W$ in $\RR^{d\times d}$ are available to us
(and hence so are $P_{U^\perp},P_{V^\perp},P_{W^\perp}$ and $P_Z$,
via \cref{ex:perp} and \cref{c:AD3}, respectively).

\subsection{Ryu splitting}

In this subsection, we consider Ryu splitting.
First, 
the block matrix representation of the operator $M$ occurring 
in Ryu splitting (see \cref{e:linM}) is 
\begin{equation}
\begin{bmatrix}
P_U \;&  0 \\[0.5em]
P_VP_U \;& P_V\\[0.5em]
P_WP_U+P_WP_VP_U-P_W\;\; & P_WP_V-P_W
\end{bmatrix}\in\RR^{3d\times 2d}.
\label{e:RyuMmat}
\end{equation}
Hence, using \cref{e:linT}, 
we obtain the following matrix representation of the
Ryu splitting operator $T=\tryu$:
\begin{subequations}
\begin{align}
T &=\textcolor{black}{\begin{bmatrix}\Id\;&0\\[0.5em]0\;&\Id\end{bmatrix}+}
\begin{bmatrix}
-\Id & 0 &\Id \\[0.5em]
0 & -\Id & \Id
\end{bmatrix}
\begin{bmatrix}
P_U \;&  0 \\[0.5em]
P_VP_U \;& P_V\\[0.5em]
P_WP_U+P_WP_VP_U-P_W\;\; & P_WP_V-P_W
\end{bmatrix}\\[1em]
&= 
\begin{bmatrix}
\textcolor{black}{\Id}-P_U+P_WP_U+P_WP_VP_U-P_W &\;\; P_WP_V-P_W\\[0.5em]
P_WP_U+P_WP_VP_U-P_W-P_VP_U & \;\; \textcolor{black}{\Id+}P_WP_V-P_V-P_W
\end{bmatrix}\in\RR^{2d\times 2d}. 
\end{align}
\end{subequations}
Next, we set, as in \cref{firelemma},
\begin{subequations}
\begin{align}
\Delta &= \menge{[x,x]^\intercal\in \RR^{2d}}{x\in X},\\
E&=\big(U^\perp\times V^\perp) \cap \big(\Delta^\perp+(\{0\}\times W^\perp) \big)\label{e:RyuE}
\end{align}
\end{subequations}
so that, by \cref{e:210815c},
\begin{equation}
\label{e:RyuE4}
P_{\Fix T}\begin{bmatrix}x\\y\end{bmatrix} = 
\begin{bmatrix}P_{Z}x\\0\end{bmatrix}+ P_E\begin{bmatrix}x\\y\end{bmatrix}.
\end{equation}
With the help of \cref{c:AD3}, we see that the first term, $[P_{Z}x,0]^\intercal$, 
is obtained by applying the matrix 
\begin{equation}
\label{e:RyuE1}
\begin{bmatrix}
P_Z & 0 \\
0 & 0 
\end{bmatrix}
= 
\begin{bmatrix}
4P_U(P_U+P_V)^\dagger P_V\big(2P_U(P_U+P_V)^\dagger P_V+P_W\big)^\dagger P_W & 0 \\
0 & 0
\end{bmatrix}
\in\RR^{2d\times 2d}
\end{equation}
to $[x,y]^\intercal$. 
Let's turn to $E$, which is an intersection of two linear subspaces.
The projector of the left linear subspace making up this intersection, $U^\perp\times V^\perp$, 
has the matrix representation
\begin{equation}
\label{e:RyuE3}
P_{U^\perp \times V^\perp} = 
\begin{bmatrix}
\Id-P_U & 0 \\
0 & \Id-P_V
\end{bmatrix}. 
\end{equation}
We now turn to the right linear subspace, 
$\Delta^\perp+(\{0\}\times W^\perp)$, which is a sum
of two subspaces whose complements are
$\Delta^{\perp\perp}=\Delta$ and
$((\{0\}\times W^\perp)^\perp = X\times W$, respectively.
The projectors of the last two subspaces are
\begin{equation}
P_\Delta =
\frac{1}{2}
\begin{bmatrix}
\Id & \Id \\
\Id & \Id
\end{bmatrix}
\;\;\text{and}\;\;
P_{X\times W} =
\begin{bmatrix}
\Id & 0 \\
0 & P_W
\end{bmatrix},
\end{equation}
respectively.
Thus, \cref{c:ADsum} yields
\begin{subequations}
\label{e:RyuE2}
\begin{align}
&P_{\Delta^\perp+(\{0\}\times W^\perp)}\\
&= 
\begin{bmatrix}
\Id & 0 \\
0 & \Id
\end{bmatrix}
- 2\cdot \frac{1}{2}
\begin{bmatrix}
\Id & \Id \\
\Id & \Id
\end{bmatrix}
\left(
\frac{1}{2}
\begin{bmatrix}
\Id & \Id \\
\Id & \Id
\end{bmatrix}
+
\begin{bmatrix}
\Id & 0 \\
0 & P_W
\end{bmatrix}
\right)^\dagger
\begin{bmatrix}
\Id & 0 \\
0 & P_W
\end{bmatrix}\\
&= 
\begin{bmatrix}
\Id & 0 \\
0 & \Id
\end{bmatrix}
- 2
\begin{bmatrix}
\Id & \Id \\
\Id & \Id
\end{bmatrix}
\begin{bmatrix}
{3}\Id & \Id \\
\Id & \Id+2P_W
\end{bmatrix}^\dagger
\begin{bmatrix}
\Id & 0 \\
0 & P_W
\end{bmatrix}. 
\end{align}
\end{subequations}
To compute $P_E$, where $E$ is as in \cref{e:RyuE}, 
we combine \cref{e:RyuE3}, \cref{e:RyuE2} under the umbrella
of \cref{f:AD} --- the result does not seem to simplify so 
we don't typeset it. 
Having $P_E$, we simply add it to \cref{e:RyuE1} to obtain 
$P_{\Fix T}$ because of \cref{e:RyuE4}.

\subsection{Malitsky-Tam splitting}

In this subsection, we turn to Malitsky-Tam splitting for the current setup
--- this corresponds to \cref{ss:MTlin} with $n=3$ and where
we identify $(U_1,U_2,U_3)$ with $(U,V,W)$.

The block matrix representation of $M$ from \cref{e:linMTM} is 
\begin{equation}
\begin{bmatrix}
P_U \;&  0 \\[0.5em]
-P_V(\Id-P_U) \;& P_V\\[0.5em]
P_W(P_U+P_VP_U-P_V)\;\; & -P_W(\Id-P_V)
\end{bmatrix}
\in\RR^{3d\times 2d}.
\label{e:MTMmat}
\end{equation}
Thus, using \cref{e:linMT}, 
we obtain the following matrix representation 
of the Malitsky-Tam splitting operator $T=\tmt$: 
\begin{subequations}
\begin{align}
T &=
\textcolor{black}{\begin{bmatrix}\Id\;&0\\[0.5em]0\;&\Id\end{bmatrix}+}
\begin{bmatrix}
-\Id & \Id & 0 \\[0.5em]
0 & -\Id & \Id
\end{bmatrix}
\begin{bmatrix}
P_U \;&  0 \\[0.5em]
-P_V(\Id-P_U) \;& P_V\\[0.5em]
P_W(P_U+P_VP_U-P_V)\;\; & -P_W(\Id-P_V)
\end{bmatrix}\\[1em]
&= 
\begin{bmatrix}
\textcolor{black}{\Id}-P_U-P_V(\Id-P_U) &\;\; P_V\\[0.5em]
P_V(\Id-P_U)+P_W(P_U+P_VP_U-P_V) & \;\; \textcolor{black}{\Id}-P_V-P_W(\Id-P_V)
\end{bmatrix}\\[1em]
&= 
\begin{bmatrix}
(\Id-P_V)(\Id-P_U) &\;\; P_V\\[0.5em]
(\Id-P_W)P_V(\Id-P_U)+P_WP_U & \;\; (\Id-P_W)(\Id-P_U)
\end{bmatrix}\in\RR^{2d\times 2d}. 
\end{align}
\end{subequations}
Next, in view of \cref{e:210817d}, 
we have 
\begin{equation}
\label{e:zahn4}
P_{\Fix T}
= \frac{1}{2}
\begin{bmatrix}
P_Z & P_Z \\
P_Z & P_Z
\end{bmatrix}
+ P_E,
\end{equation}
where (see \cref{e:bloodyE} and \cref{e:bloodyPsi})
\begin{equation}
\label{e:zahn0}
E = \ran\Psi \cap (X\times W^\perp)
\end{equation}
and 
\begin{equation}
\Psi \colon U^\perp \times V^\perp \to X^2
\colon \begin{bmatrix} y_1\\y_2\end{bmatrix}
\mapsto \begin{bmatrix} y_1\\y_1+y_2\end{bmatrix}. 
\end{equation}
We first note that 
\begin{equation}
\ran \Psi = \ran 
\begin{bmatrix}
\Id & 0 \\
\Id & \Id
\end{bmatrix}
\begin{bmatrix}
P_{U^\perp} & 0 \\
0 & P_{V^\perp}
\end{bmatrix}
=\ran
\begin{bmatrix}
P_{U^\perp} & 0 \\
P_{U^\perp} & P_{V^\perp}
\end{bmatrix}. 
\end{equation}
We thus obtain from \cref{f:Pran} that 
\begin{equation}
\label{e:zahn1}
P_{\ran\Psi} = \begin{bmatrix}
P_{U^\perp} & 0 \\
P_{U^\perp} & P_{V^\perp}
\end{bmatrix}
\begin{bmatrix}
P_{U^\perp} & 0 \\
P_{U^\perp} & P_{V^\perp}
\end{bmatrix}^\dagger.
\end{equation}
On the other hand,
\begin{equation}
\label{e:zahn2}
P_{X\times W^\perp}
= \begin{bmatrix}
\Id & 0 \\
0 & P_W^\perp
\end{bmatrix}
\end{equation}
In view of \cref{e:zahn0}
and \cref{f:AD}, we obtain
\begin{equation}
\label{e:zahn3}
P_E = 2P_{\ran\Psi}\big( P_{\ran\Psi}+P_{X\times W^\perp}\big)^\dagger
P_{X\times W^\perp}.
\end{equation}
We could 
now use our formulas \cref{e:zahn1} and \cref{e:zahn2}
for $P_{\ran\Psi}$ and $P_{X\times W^\perp}$
to obtain a more explicit formula for $P_E$ --- 
but we refrain from doing so as the expressions become unwieldy.
Finally, plugging the formula for $P_Z$ from \cref{c:AD3} 
into \cref{e:zahn4} as well as plugging \cref{e:zahn3} into \cref{e:zahn4} 
yields a formula for $P_{\Fix T}$.

\section{Numerical experiments}

\label{sec:numexp}

We now outline a few experiments conducted to observe the performance of the algorithms outlined in Section~\cref{sec:matrix}. 
Each instance of an experiment involves 3 subspaces $U_i$ of dimension $d_i$ for $i\in\{1,2,3\}$ in $X=\RR^d$. 
By \cite[equation~(4.419) on page~205]{Meyer}, 
\begin{equation}
\dim(U_1+ U_2)=d_1+d_2-\dim(U_1\cap U_2).
\end{equation}
Hence
\begin{equation}
\dim(U_1\cap U_2)=d_1+d_2-\dim(U_1+U_2)\geq d_1+d_2-d.
\end{equation}
Thus $\dim(U_1\cap U_2)\geq 1$ whenever 
\begin{equation}\label{e:U1U2dim}
d_1+d_2\geq d+1.
\end{equation}
Similarly, 
\begin{equation}
\dim(Z)\geq \dim(U_1\cap U_2)+d_3-d\geq d_1+d_2-d+d_3-d=d_1+d_2+d_3-2d.
\end{equation}
Along with \cref{e:U1U2dim}, a sensible choice for $d_i$ satisfies 
\begin{equation}
d_i\geq 1+\lceil2d/3\rceil
\end{equation}
because then $d_1+d_2\geq 2+2\lceil 2d/3\rceil\geq 2+4d/3>2+d$. 
Hence $d_1+d_2\geq 3+d$ and $d_1+d_2+d_3> 3+3\lceil 2d/3\rceil\geq 3+2d$. 
The smallest $d$ that gives proper subspaces 
is $d=6$, for which $d_1=d_2=d_3=5$ satisfy the above conditions.

We now describe our set of 3 numerical experiments designed to 
observe different aspects of the algorithms. 

\subsection{Experiment 1: Bounds on the rates of linear convergence}

\begin{figure}[htbp!]
\centering
\includegraphics[width=0.8\textwidth]{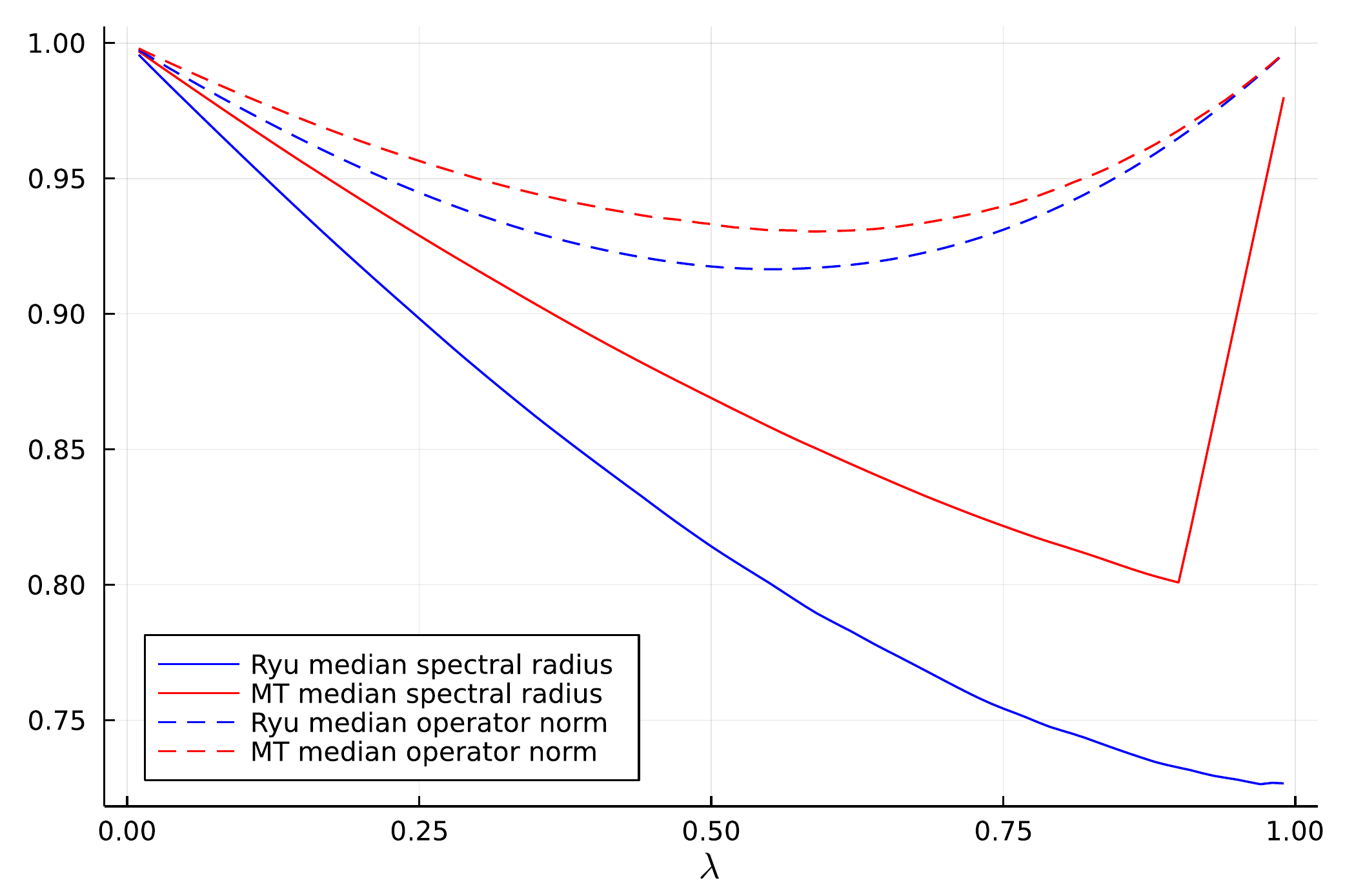}
\caption{Experiment 1: spectral radii and operator norms}
\label{fig:exp1}
\end{figure}

As shown in \cref{sec:matrix}, we have lower and upper bounds on the rate of 
linear convergence of the operator $T_\lambda$. 
We conduct this experiment to observe how these bounds change as we increase 
$\lambda$.  To this end, we 
generate 1000 instances of sets of linear subspaces $U_1,U_2$ and $U_3$. 
This can be done by randomly generating sets of 3 matrices 
$B_1, B_2, B_3$ in $\RR^{5\times 6}$. 
These can be used to define the range spaces of these subspaces, 
which in turn will give us the projection onto $U_i$ using \cite[Proposition~3.30(ii)]{BC2017},
\begin{equation}
P_{U_i}=B_iB_i^\dagger.
\end{equation}
For each instance, algorithm and $\lambda\in 
\menge{0.01\cdot k}{k\in\{1,2,\ldots,99\}}$, 
we obtain the operators $T_\lambda$ and $P_{\Fix T}$ as outlined in \cref{sec:matrix} 
and compute the spectral radius and operator norm of $T_\lambda-P_{\Fix T}$. 
\cref{fig:exp1} reports the average of the spectral radii and operator norms for each $\lambda$. While Ryu sees a decline in the lower bound for the rate of convergence, MT sees a minimizer around $0.9$.

\subsection{Experiment 2: Number of iterations to achieve prescribed accuracy}

\begin{figure}[htbp!]
\centering
\includegraphics[width=0.8\textwidth]{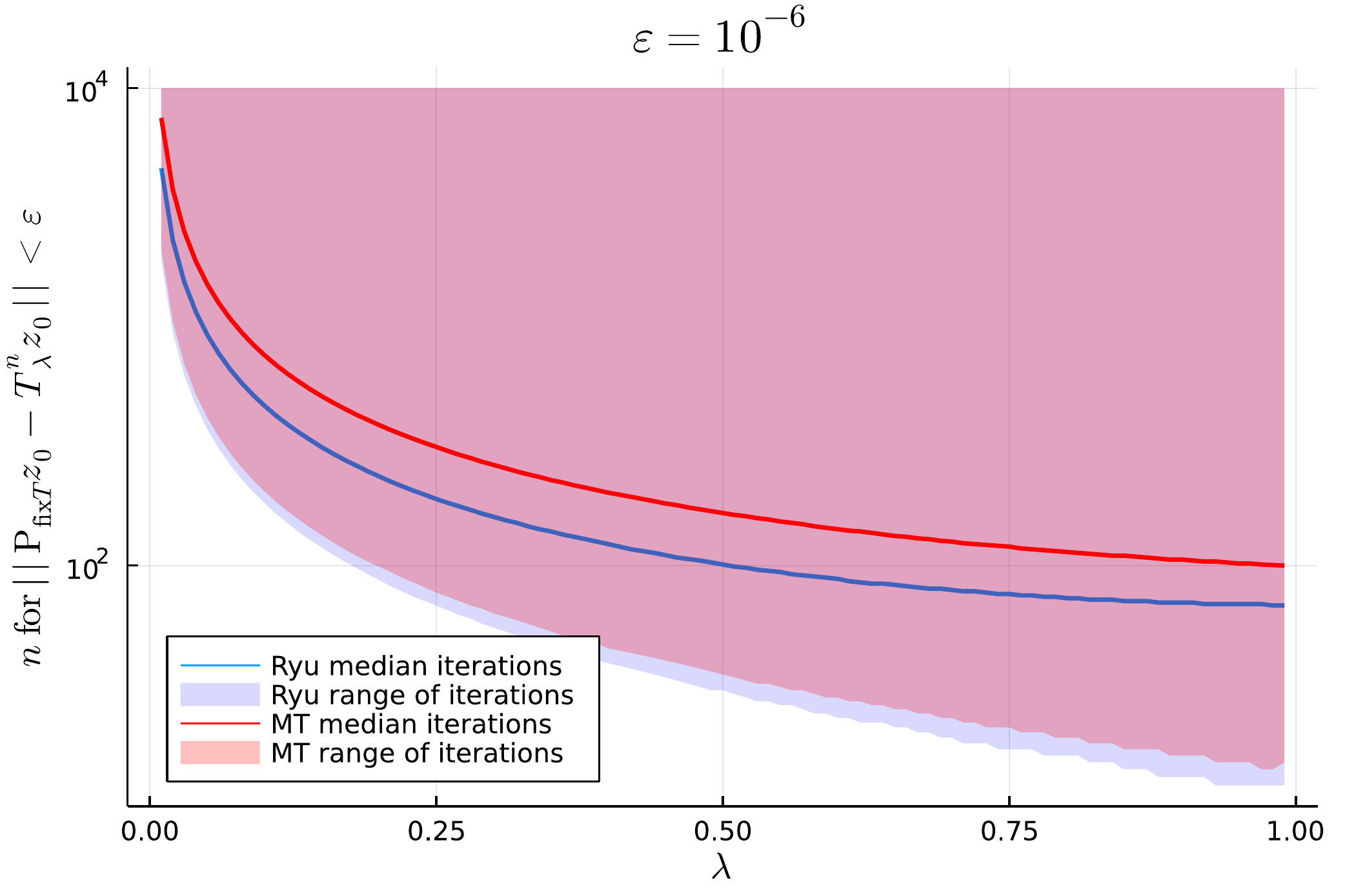}
\caption{Experiment 2: number of iterations for the governing sequence}
\label{fig:exp2a}
\end{figure}

\begin{figure}[htbp!]
\centering
\includegraphics[width=0.8\textwidth]{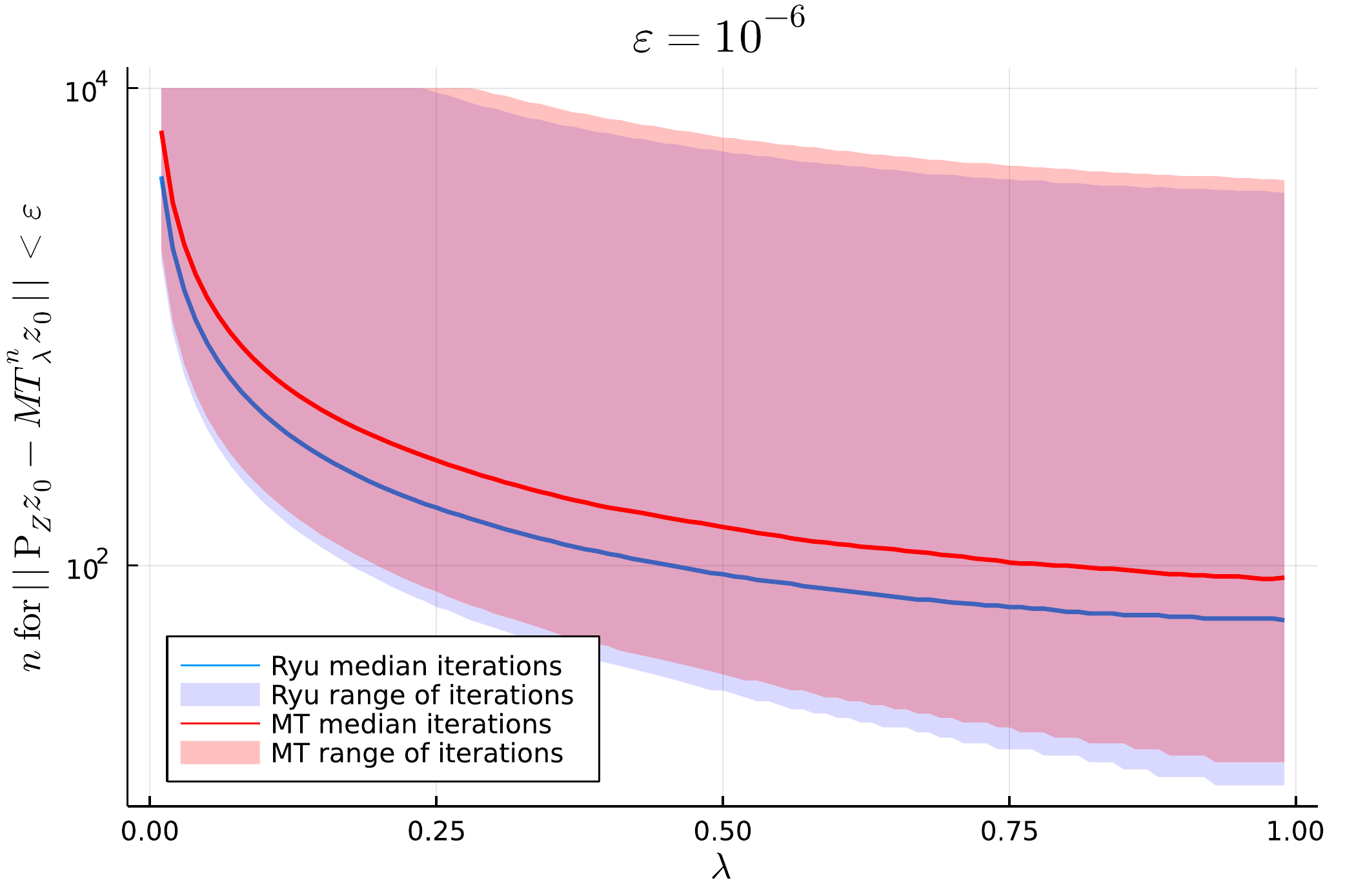}
\caption{Experiment 2: number of iterations for the shadow sequence}
\label{fig:exp2b}
\end{figure}

Because we know 
the limit points of the governing as well as shadow sequences, 
we investigate how changing $\lambda$ affects the number of iterations 
required to approximate 
the limit to a given accuracy.
For these experiments, we fix 100 instances of sets of subspaces $\{U_1,U_2,U_3\}$. 
We also fix 100 different starting points in $\RR^6$.
For each instances of the subspaces, starting point $z_0$ and 
$\lambda\in\menge{0.01\cdot k}{k\in\{1,2,\ldots,99\}}$, 
we obtain the number of iterations (up to a maximum of $10^4$ iterations) required to
achieve $\varepsilon=10^{-6}$ accuracy.

For the governing sequence, the limit $P_{\fix T}z_0$ is used to determine the stopping condition. 
\cref{fig:exp2a} reports the median number of iterations required for each $\lambda$ to achieve the given accuracy. 
For the shadow sequence, we compute the median number of iterations required to achieve $\varepsilon=10^{-6}$ accuracy for the shadow sequence $Mz_k$ with respect to its limit $(P_Zz_0,P_Zz_0,P_Zz_0)$. 
Here $M$ for Ryu and MT can be obtained from \cref{e:RyuMmat} and \cref{e:MTMmat} respectively. See \cref{fig:exp2b} for results. 

For both the algorithms and experiments, 
increasing values of $\lambda$ result in a decreasing number of median iterations required. 
As is evident from the maximum number of iterations required for a fixed lambda, 
the shadow sequence converges before the governing sequence for larger values of $\lambda$.
One can also see that Ryu requires fewer median iterations 
for both the governing and the shadow sequence to achieve the same accuracy as MT for a fixed lambda.

\subsection{Experiment 3: Convergence plots of shadow sequences}

\begin{figure}[htbp!]
\centering
\includegraphics[width=0.8\textwidth]{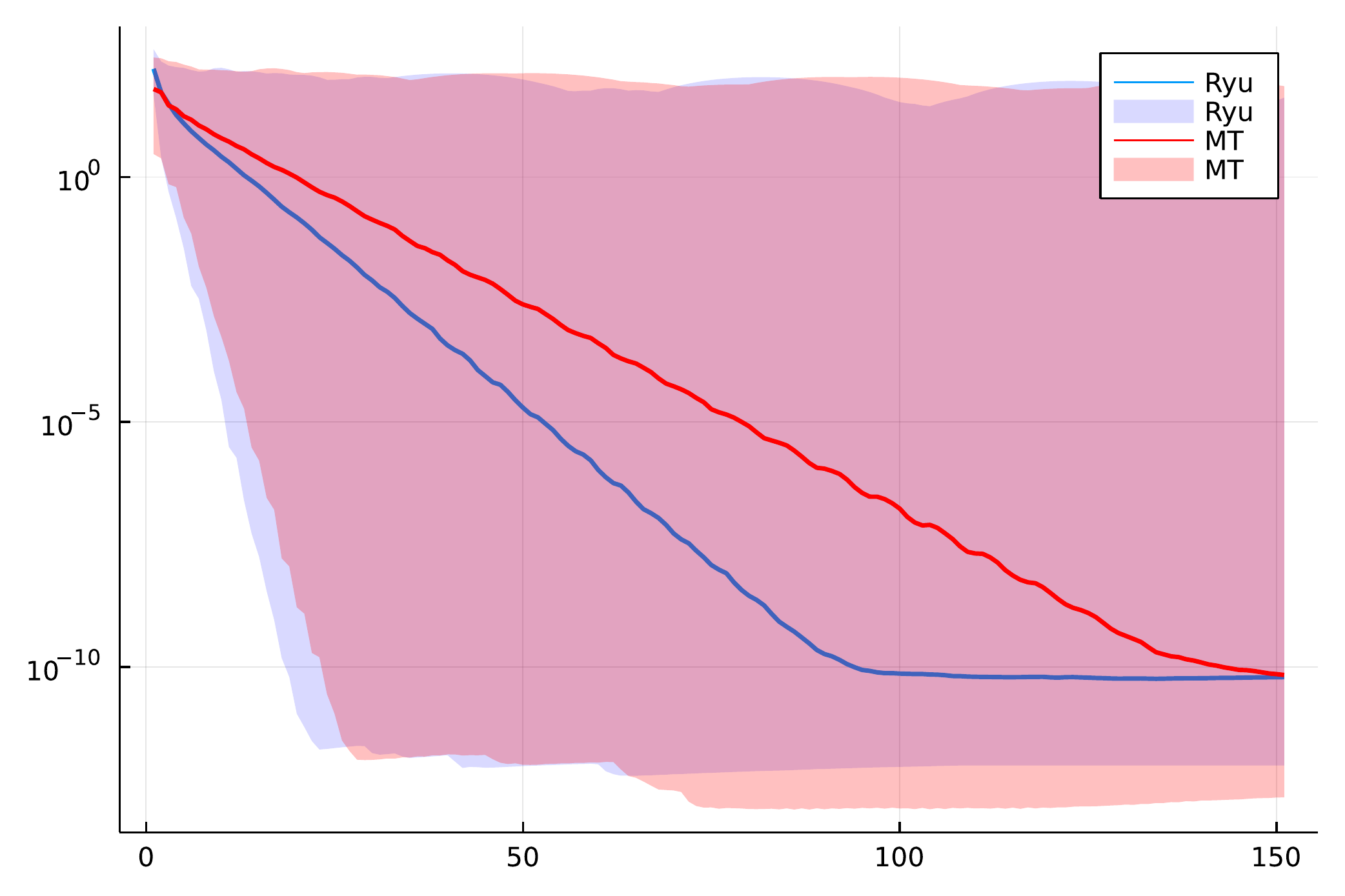}
\caption{Experiment 3: convergence plot of the shadow sequence}
\label{fig:exp3}
\end{figure}

In this experiment, we measure the distance of the shadow sequence from the limit point 
for each iteration to observe the approach of the iterates of the algorithm to the solution. 
We pick the $\lambda$ with respect to which the iterates converge the fastest, 
which is $\lambda=0.99$ for both the algorithms because of \cref{fig:exp2b}. 
Similar to the setup of the previous experiment, 
we fix 100 starting points and 100 sets of subspaces $\{U_1,U_2,U_3\}$. 
We now run the algorithms for 150 iterations for each starting point and each set of subspaces. 
We measure $\norm{Mz_n-(P_Zz_0,P_Zz_0,P_Zz_0)}$ 
for each iteration. 
\cref{fig:exp3} reports the median of $\norm{Mz_i-(P_Zz_0,P_Zz_0,P_Zz_0)}$ for each iteration $i\in\{1,\dots,150\}$.

As can be seen in \cref{fig:exp3}, Ryu converges faster to the solution compared to MT. Both show faint ``rippling'' akin to the one known to occur for the Douglas-Rachford algorithm. 

\section{Conclusion}

\label{sec:end}

In this paper, we investigated the recent splitting methods 
by Ryu and by Malitsky-Tam in the context of normal cone operators
for subspaces. We discovered that both algorithms find not just 
some solution but in fact the \emph{projection} of the starting point onto the
intersection of the subspaces. Moreover, convergence of the iterates 
is \emph{strong} even in infinite-dimensional settings. 
Our numerical experiments illustrated that Ryu's method seems to converge 
faster although Malitsky-Tam splitting is not limited in its applicability to
just 3 subspaces. 

Two natural avenues for future research are the following.
Firstly, when $X$ is finite-dimensional, we know that the convergence
rate of the iterates
is linear. While we illustrated this linear convergence
numerically in this paper, it is open whether there are \emph{natural
bounds for the linear rates} in terms of some version of angle between
the subspaces involved. 
For the prototypical Douglas-Rachford splitting framework, 
this was carried out in  \cite{BBCNPW1} in terms of the \emph{Friedrichs angle}.
Secondly, what can be said in the \emph{inconsistent} affine case?
Again, the Douglas-Rachford algorithm may serve as a guide to what 
the expected results and complications might be; see, e.g., \cite{BM16}.

\end{document}